\documentclass[a4paper, 12pt]{article}
\usepackage{latexsym}
\usepackage{amsmath}
\usepackage{amsthm}
\usepackage{amssymb}
\usepackage[all]{xy}

\numberwithin{equation}{section}

\newtheorem{thm}[equation]{Theorem}
\newtheorem{prop}[equation]{Proposition}
\newtheorem{lem}[equation]{Lemma}
\newtheorem{cor}[equation]{Corollary}
\newtheorem{conj}[equation]{Conjecture}
\newtheorem{question}[equation]{Question}

\theoremstyle{definition}
\newtheorem{defn}[equation]{Definition}

\theoremstyle{remark}
\newtheorem{rem}[equation]{Remark}

\newcommand{\N}{\mathbb{N}}

\newcommand{\arc}{\rightarrow}

\newcommand{\Sym}{\mathrm{Sym}}
\newcommand{\Aut}{\mathrm{Aut}}
\newcommand{\Alt}{\mathrm{Alt}}

\begin{document}

\author{D.\ C.\ Lockett, \\School of Mathematics,\\ University of Leeds,\\ Leeds LS2 9JT, UK  
\and H.\ D.\ Macpherson (corresponding author),\\ School of Mathematics, \\University of Leeds,\\ Leeds LS2 9JT, UK, \\h.d.macpherson@leeds.ac.uk, \\tel.+441133435166, fax +441133435090}

\title{Orbit-equivalent infinite permutation groups}

\maketitle

\begin{abstract}
Let $G,H$ be closed permutation groups on an infinite set $X$, with $H$ a subgroup of $G$. It is shown that if $G$ and $H$ are {\em orbit-equivalent}, that is, have the same orbits on the collection of finite subsets of $X$, and $G$ is primitive but not 2-transitive, then $G=H$.
\end{abstract}


\section{Introduction}

We consider closed permutation groups acting on an infinite set $X$; that is, subgroups of $\Sym(X)$ which are closed in $\Sym(X)$ in the topology of \emph{pointwise convergence} on $\Sym(X)$ with respect to the discrete topology on $X$ (so the basic open sets are cosets of pointwise stabilisers of finite sets). It is easily checked that a closed permutation group on $X$ is precisely the automorphism group of a relational structure with domain $X$.
Two permutation groups $G,H$ on the set $X$  are said  to be {\em orbit-equivalent} if, for every positive integer $k$, $G$ and $H$ have the same orbits on the collection  of unordered $k$-element subsets of $X$, denoted here by $X^{[k]}$. This generalises a definition for finite permutation groups. Observe that if $G,H$ are orbit-equivalent, then they are each orbit-equivalent to $\langle G,H\rangle$. Thus, to investigate such pairs, it suffices to consider $G,H$ with $H$ a subgroup of $G$. Easily, if $H\leq G$ and $G,H$ are orbit-equivalent, then $G$ is transitive (on $X$) if and only if $H$ is transitive, and also $G$ and $H$ preserve the same systems of imprimitivity on $X$; so $G$ is primitive on $X$ (that is, preserves no proper non-trivial equivalence relation on $X$) if and only if $H$ is primitive.

Our main theorem is the following. Our particular interest is in the case when $X$ is countably infinite, but the proofs below do not use countability.

\begin{thm}\label{mainthm}
Let $G,H$ be orbit-equivalent closed permutation groups on the infinite set
$X$, with $H\leq G$, and suppose that $G$ is primitive but not 2-transitive. Then $H=G$.
\end{thm}

We stress that if $H$ is a closed proper subgroup of $G\leq \Sym(X)$, then for some $k>0$, some $G$-orbit on $X^k$ (the set of $k$-tuples from $X$) breaks into more than one $H$-orbit. 
The assumption in the theorem that $G$ and $H$ are closed seems essential; indeed, any subgroup $H$ of $\Sym(X)$ is orbit-equivalent to its closure, and, for example, the dense (and so orbit-equivalent) subgroups of $\Sym(X)$ are exactly the subgroups of $\Sym(X)$ which are $k$-transitive for all positive integers $k$, and these seem hopelessly unclassifiable.

This paper takes its motivation from two sources. First, there is an extended literature on  primitive orbit-equivalent pairs of permutation groups
 on a {\em finite} set $X$; see for example \cite{siemons, inglis, siemwag}. Clearly, the symmetric and alternating groups $\Sym_n$ and $\Alt_n$, in their natural actions on $\{1,\ldots,n\}$, are orbit-equivalent for $n\geq 3$.
 Also, if $G$ is a permutation group on a finite set $X$ and has a regular orbit $U$ on the power set
${\cal P}(X)$, and $H$ is a proper subgroup of $G$, then $H$ is intransitive on $U$, and so $H$ is not orbit-equivalent to $G$. It is shown in \cite{cns} that if $X$ is finite then there are just finitely many primitive subgroups	of $\Sym(X)$ which do not contain $\Alt(X)$ and have no regular orbit on ${\cal P}(X)$ (and so {\em could} have an orbit-equivalent proper subgroup). Such primitive groups $G$ (with no regular orbit on $X$) are  classified by Seress in \cite{seress}, who then classifies all pairs of finite primitive orbit-equivalent permutation groups $(H,G)$ with $H<G$.	There is further work on the finite imprimitive case in 
\cite{seressyang}.																			

The second source of motivation is more model-theoretic, namely the study of homogeneous structures. Recall that a countable (possibly finite) structure $M$ in a first order relational language is said to be {\em homogeneous} if every isomorphism between finite substructures of $M$ extends to an automorphism of $M$. A natural generalisation, originally considered by Fra\"iss\'e in \cite{fra}, is to say that the countable structure $M$ is  	{\em set-homogeneous} if, whenever $U,V$ are isomorphic finite substructures of $M$, there is $g\in\Aut(M)$ with $U^g=V$. Finite set-homogeneous graphs are classified by Ronse in \cite{ronse}, and  a very short proof was given by Enomoto in \cite{en} that every finite set-homogeneous graph is homogeneous. There is a classification of set-homogeneous digraphs (allowing two vertices to be linked by an arc in each direction) in \cite{gmpr}, building on a corresponding classification of finite homogeneous digraphs by Lachlan \cite{lachlan}. Also, there are initial results on countably infinite set-homogeneous structures, in particular graphs and digraphs, in \cite{sethomg} and \cite{gmpr}. 	The latter paper  poses the following related question: given a homogeneous structure $M$, when does $\Aut(M)$ have a proper closed subgroup $H$ which acts set-homogeneously on $M$, that is, has the same orbits as $\Aut(M)$ on the collection of unordered finite subsets of $M$? Equivalently, for which $M$ does $\Aut(M)$ have a proper closed orbit-equivalent subgroup? (Here, and throughout the paper, we use the same symbol $M$ for a structure and for its domain.)

A countably infinite set $X$ in the empty language is homogeneous, and has automorphism group $\Sym(X)$. By a theorem of Cameron \cite{cam}, $\Sym(X)$ has just four orbit-equivalent closed proper subgroups, namely $\Aut(X,<)$, $\Aut(X,B)$, $\Aut(X,C)$, and $\Aut(X,S)$. Here $<$ is a  dense linear order without end points on $X$, $B$ is the (ternary) linear betweenness relation on $X$ induced from $<$, $C$ is the (also ternary) circular order on $X$ induced from $<$, and $S$ is the corresponding arity 4 separation relation. Observe that $\Aut(X,S)=\langle \Aut(X,B),\Aut(X,C)\rangle$ and is 3-transitive but not 4-transitive.	Our conjecture below would strengthen Theorem~\ref{mainthm} by removing the `not 2-transitive' assumption.

\begin{conj} \label{mainconj} Let $G$ and $H$ be distinct orbit-equivalent  primitive closed permutation groups on a countably infinite set $X$. Then $G$ and $H$ belong to the list $\Aut(X,<)$, $\Aut(X,B)$, $\Aut(X,C)$, $\Aut(X,S)$, $\Sym(X)$ described above.
\end{conj}																					Recall the following standard terminology, for a permutation group $G$ on  a set $X$, and an integer $k>0$:
$G$ is \emph{$k$-transitive} if it is transitive on the ordered $k$-subsets of $X$;
and $G$ is \emph{$k$-homogeneous} if it is transitive on the unordered $k$-subsets of $X$.
 Also, if $U$ is a subset of $X$ then $G_{\{U\}}$ and $G_{(U)}$ denote respectively the setwise and pointwise stabilisers of $U$ in $G$, and if $x\in X$ then
$G_x:=\{g\in G:x^g=x\}$.

The proof of Theorem~\ref{mainthm} splits into two cases: 

(1) $G$ is primitive but not $2$-homogeneous; 

(2) $G$ is $2$-homogeneous (and so primitive) but is not $2$-transitive.

Our main tool for both cases is the notion of \emph{local rigidity}.
We shall say that a permutation group $G$ acting on an infinite set $X$ acts \emph{locally rigidly} if for all finite $U \subset X$, there is some finite $V \subset X$ such that $U \subseteq V$ and the setwise stabiliser $G_{\{V\}}$ of $V$ fixes $U$ pointwise. Likewise, a first order relational structure $M$ is {\em locally rigid} if, for every finite substructure $U$ of $M$, there is a finite substructure $V$ of $M$ containing $U$ such that every automorphism of $V$ fixes $U$ pointwise. Clearly, if a relational structure $M$ is locally rigid, then any subgroup of  its automorphism group acts locally rigidly on $M$. Strengthening the notion of local rigidity, we shall later say that a countably infinite first order structure $M$ is {\em cofinally rigid} if, for every finite substructure $U$ of $M$, there is a finite substructure $V$ of $M$ with $U\subseteq V$ such that the automorphism group of  $V$ is trivial. Here, `substructure' is used in the standard model-theoretic sense, corresponding to the graph-theoretic notion of `induced subgraph'.

\begin{lem} \label{locallyrigid}
Let $G,H$ be closed permutation groups on $X$, with $H \le G$. If $G$ and $H$ are orbit-equivalent and $G$ acts locally rigidly, then $H=G$.
\end{lem}

\begin{proof}
It suffices to show that $H$ has the same orbits as $G$ on $X^k$ for all $k$. 
So let $\overline{u}_1, \overline{u}_2 \in X^k$ be in the same orbit of $G$; that is, there is $g \in G$ such that $\overline{u}_1^g = \overline{u}_2$.
Let $U_1, U_2 \subset X$ be enumerated by $\overline{u}_1, \overline{u}_2$ respectively.
Since $G$ acts locally rigidly on $X$, there is finite $V_1 \subset X$ such that $U_1 \subseteq V_1$ and $G_{\{V_1\}} \le G_{(U_1)}$.
Let $V_2 := V_1^g$. 
Then $V_1, V_2$ are in the same orbit of $G$, so by orbit-equivalence there is some $h \in H$ such that $V_1^h = V_2$. 
Now $gh^{-1} \in G_{\{V_1\}}$, so in fact $gh^{-1} \in G_{(U_1)}$.
Thus $\overline{u}_1^h = \overline{u}_2$ as required.
\end{proof}

In both cases (1) and (2) ($G$ primitive, and either not $2$-homogeneous, or $2$-homogeneous but not $2$-transitive) we shall show that $G$ acts locally rigidly on $X$. In fact, in the second case we show that $G$ is a group of automorphisms of a {\em cofinally rigid} tournament. 
Our method to show the local rigidity of such actions stems from a similar result in \cite{sethomg}, which we adapt. 
Formally, we view a graph $\Gamma$ as a relational structure $\Gamma=(X,R)$, where $R$ is a symmetric irreflexive binary relation on $X$. 
Given  a graph $\Gamma$, if $x,y$ are vertices we write $x\sim y$ if $x$ and $y$ are adjacent, and  let $\Gamma(x):= \{ v \in X: v \sim x \}$, the \emph{neighbour set} of $x$. We shall prove in Lemma~\ref{orderedgraph} a strengthening of the following result.

\begin{lem}\cite{sethomg} \label{prop6.1}
Let $\Gamma$ be an infinite graph such that, for all distinct vertices $x,y$ of $\Gamma$, the sets  $\Gamma(x) \setminus \Gamma(y)$ and $\Gamma(y) \setminus \Gamma(x)$ are both infinite. Then $\Gamma$ is locally rigid. 
\end{lem}

We draw attention to a basic Ramsey-theoretic principle which is well-known, for example in model theory, and used below in both the primitive not 2-homogeneous case, and the 2-homogeneous not 2-transitive case. 

\begin{defn} \label{ramsey1}
Let $L$ be a finite relational language, let $M$ be a first order $L$-structure, $A$ a finite subset of the domain of $M$, and $P_1,\ldots,P_r$ 
disjoint subsets of $M\setminus A$, with $P_i:=\{p_{i,0},\ldots,p_{i,n-1}\}$ for each $i=1,\ldots,r$. We say that $P_1\ldots,P_r$ are {\em mutually indiscernible over $A$}
if the following holds for any positive integers $e_1,\ldots, e_r<n$: for each $j=1,\ldots,r$, let $\bar{p}_j,\bar{p}_j'$ be $e_j$-tuples from $P_j$, each listed in increasing order (so if $\bar{p}_j=(p_{j,i(1)},\ldots,p_{j,i(e_j)})$, then 
$i(1)<\ldots<i(e_j)$); then the map taking $\bar{p}_j$ to $\bar{p}_j'$ for each $j$, extended by the identity on $A$, is an isomorphism of $L$-structures. 
\end{defn}

\begin{lem} \label{ramsey2}
Let $M,L,A$ be as in Definition~\ref{ramsey1} with $M$ infinite, and let 
$Q_1,\ldots,Q_r$ be countably infinite disjoint subsets of $M\setminus A$. Let $n$ be a positive integer.  Then the following hold.

(i) There are subsets $P_1\subset Q_1,\ldots,P_r\subset Q_r$, each of size $n$, such that $P_1,\ldots,P_r$ are mutually indiscernible over $A$ (with respect to some indexing of each $P_i$).

(ii) If every relation of $L$ is of arity at most 2, and $P_1,\ldots,P_r$ are as in (i), then
for each $i=1,\ldots,r$, either some relation of $L$ induces a total order on $P_i$, or every permutation of $P_i$, extended by the identity on $S_i:=A \cup \bigcup_{j\neq i}P_j$, is an automorphism of the induced $L$-structure on $S:=A \cup P_1 \cup \ldots \cup P_r$.
\end{lem}

\begin{proof} (Sketch) 
(i) Let $Q_i:=\{q_{i,j}:j\in {\mathbb N}\}$ for each $i=1,\ldots,r$.
Let $d$ be the maximum arity of a relation in $L$. Colour each subset
$\{i_1,\ldots,i_d\}$ of ${\mathbb N}$  in such a way that given natural numbers $i_1<\ldots<i_d$ and $k_1<\ldots<k_d$,
 the map
$$(q_{1,i_1},\ldots,q_{1,i_d},\ldots,q_{r,i_1},\ldots,q_{r,i_d})\mapsto
(q_{1,k_1},\ldots,q_{1,k_d},\ldots,q_{r,k_1},\ldots,q_{r,k_d})$$
is an isomorphism over $A$ if and only if
$\{i_1,\ldots,i_d\}$ and $\{k_1,\ldots,k_d\}$ have the same colour. By Ramsey's Theorem, replacing ${\mathbb N}$ by an infinite monochromatic subset if necessary, we may suppose that ${\mathbb N}$ is monochromatic. Now let $p_{i,j}:=q_{i, (i-1)n+j}$ for each $i=1,\ldots,r$ and $j=0,\ldots,n-1$. Put $P_i:=\{p_{i,1},\ldots,p_{i,n-1}\}$ for each $i=1,\ldots,r$. Then $P_1,\ldots,P_r$ are mutually indiscernible over $A$.

(ii) This is immediate from (i).

\end{proof}

The case of Theorem~\ref{mainthm} when $G$ is primitive but not 2-homogeneous is handled in Section 2, and the 2-homogeneous but not 2-transitive case is treated in Section 3. Section 4 consists of some further observations, about bounds in local rigidity, approaches to
 Conjecture~\ref{mainconj}, and regular orbits on the power set. We also observe that our proofs give a slight strengthening of Theorem~\ref{mainthm}, namely Theorem~\ref{stronger}.

\section{$G$ primitive but not $2$-homogeneous}

In this section we prove the following.

\begin{prop}\label{not2hom}
Let $G$ be a primitive but not 2-homogeneous permutation group on an infinite set $X$. Then the action of $G$ on $X$ is locally rigid.
\end{prop}

The proposition follows rapidly from the following two lemmas. The first uses an argument in \cite[Proposition 4.4]{macpherson}.
\begin{lem} \label{graph}
Let $G$ be a primitive but not 2-homogeneous permutation group on an infinite set $X$. Then there is a $G$-invariant graph $\Gamma$ with vertex set $X$ such that for all distinct $x,y\in X$, the symmetric difference $\Gamma(x)\triangle\Gamma(y)$ is infinite.
\end{lem}

\begin{proof}
Let $U$ be any $G$-orbit on the collection of 2-subsets of $X$. Then $U$ is the edge set of a  $G$-invariant graph $\Gamma_0$ with vertex set $X$, and as $G$ is not 2-homogeneous, $\Gamma_0$ is not complete. For $x\in X$, write $\Gamma_0(x)$ for the neighbour set of $x$ in $\Gamma_0$. Define the equivalence relation $\equiv_0$ on $X$, putting $x\equiv_0 y$ if and only if $|\Gamma_0(x)\triangle \Gamma_0(y)|$ is finite.
Then $\equiv_0$ is $G$-invariant, so by primitivity $\equiv_0$ is trivial or universal. The lemma holds if $\equiv_0$ is trivial, so we shall suppose that $\equiv_0$ is universal.

Recall that a graph is {\em locally finite} if all of its vertices have finite degree.

\medskip

{\em Claim.}
Either $\Gamma_0$ or its complement is locally finite.

\medskip

{\em Proof of Claim.} Suppose not, and fix $x\in X$. Then both
$\Gamma_0(x)$ and $X\setminus \Gamma_0(x)$ are infinite. If $y\in\Gamma_0(x)$ then (as $\equiv_0$ is universal) $\Gamma_0(y)\setminus \Gamma_0(x)$ is finite. Hence as $G_x$ has at most two orbits  on $\Gamma_0(x)$ there is $k\in \N$ such that for all $y\in \Gamma_0(x)$, we have
$|\Gamma_0(y)\setminus \Gamma_0(x)|\leq k$. Pick distinct $z_1,\ldots,z_{k+1}\in X\setminus (\{x\}\cup \Gamma_0(x))$. Then as $x\equiv_0 z_i$ for each $i$, each set $\Gamma_0(z_i)\cap \Gamma_0(x)$ is cofinite in $\Gamma_0(x)$. Hence there is $y\in\Gamma_0(x)\cap \bigcap_{i=1}^{k+1}\Gamma_0(z_i)$. Then $z_1,\ldots,z_{k+1}\in \Gamma_0(y)\setminus \Gamma_0(x)$, so $|\Gamma_0(y)\setminus \Gamma_0(x)|\geq k+1$, which is a contradiction.

\medskip

By the claim, replacing $\Gamma_0$ by its complement if necessary, we may suppose that $\Gamma_0$ is locally finite. By our original assumption that $\Gamma_0$ is not complete (or null), $\Gamma_0$ has an edge. By primitivity, $\Gamma_0$ is connected.
Now let $\Gamma$ be the graph on $X$ whose edge set consists of the set of unordered pairs an even distance apart in $\Gamma_0$. Then $\Gamma$ is also $G$-invariant. Pick $v_0 \in X$, and choose a $\Gamma_0$-path $v_0\sim v_1\sim v_2 \sim \ldots$ so that the distance $d_0(v_0,v_i)$ between $v_0$ and $v_i$ in $\Gamma_0$ equals $i$ for each $i$ (this is certainly possible, for example by K\"onig's Lemma). Then $v_{2i} \in \Gamma(v_0) \setminus \Gamma(v_1)$ for each $i>0$. Thus $\Gamma(v_0)$ and $\Gamma(v_1)$ have infinite symmetric difference, and since $G$ is primitive, this holds for all pairs of distinct vertices in $\Gamma$.
\end{proof}

In the next lemma, and  later in the paper, if $A,B$ are sets we write $A\subset_f B$
if $B\setminus A$ is infinite and $A\setminus B$ is finite. The lemma below extends Lemma~\ref{prop6.1}, since under the assumptions of that lemma, $x<y$ (as defined below) never holds. If $u,v,w$ are distinct vertices of the graph $\Gamma$, we say $w$ {\em separates} $u$ and $v$ if $w\in \Gamma(u)\triangle \Gamma(v)\setminus \{u,v\}$, and call a collection of such vertices $w$ a {\em separating set} for $u$ and $v$.

\begin{lem}\label{orderedgraph}
Let $\Gamma=(X,R)$ be an infinite graph, and suppose that $\Gamma(x)\triangle \Gamma(y)$ is infinite for any distinct $x,y\in X$. Write $x<y$ whenever $\Gamma(x)\supset_f\Gamma(y)$. Then the structure $\Gamma_<=(X,R,<)$ is locally rigid.
\end{lem}

\begin{proof}
We slightly adapt the proof of Proposition 6.1 from \cite{sethomg}.
So let $U=\{u_1,\ldots,u_n\}$ be a finite subset of $X$. We aim to find finite $V$ with $U \subset V \subset X$, such that $\Aut(V,R,<)$ fixes $U$ pointwise.

For each $u_i, u_j \in U$, with $i < j$, we find an infinite separating set $Q_{ij}\subset X\setminus U$ as follows: if $u_i < u_j$, then let $Q_{ij} \subset \Gamma(u_i) \setminus (\Gamma(u_j)\cup\{u_j\})$; if $u_j < u_i$, then let $Q_{ij} \subset \Gamma(u_j) \setminus (\Gamma(u_i)\cup\{u_i\})$; and if $u_i \parallel u_j$ (that is, $u_i,u_j$ are incomparable under $<$), then let $Q_{ij} \subset \Gamma(u_i) \setminus (\Gamma(u_j)\cup\{u_j\})$. 

Let $K$ be a positive integer. 
By Lemma~\ref{ramsey2} with respect to the language $L=\{R,<\}$, we can choose for each $i<j$ a subset $P_{ij}$ of $Q_{ij}$ with $|P_{ij}|=K$, such that the collection of all sets $P_{ij}$ is mutually indiscernible over $U$. Let $W = U \cup \bigcup(P_{ij}:1\leq i<j\leq n)$. 
Then 
each $P_{ij}$ carries a  complete or null induced graph structure, and for each $x,y \in P_{ij}$ and $z \in W \setminus P_{ij}$, we have $x \sim z$ if and only if $y \sim z$.

For any subset $Y$ of $X$, define the equivalence relation $\approx_Y$ on $Y$, where,
for $x,y\in Y$, $x\approx_Y y$ if and only if 
$(\Gamma(x) \triangle \Gamma(y))\cap Y\subseteq \{x,y\}$.  Then $\approx_Y$-classes always carry a complete or null (that is, independent set) induced subgraph structure. 
If $Z$ is an $\approx_Y$-class, then for $z_1,z_2\in Z$ and $y\in Y\setminus Z$, we have 
$y\sim z_1 \Leftrightarrow y\sim z_2$; in particular, $\Aut(Y)_{(Y\setminus Z)}$ induces $\Sym(Z)$.
Observe that if $Y_1\subset Y_2\subseteq X$ and $x,y\in Y_1$, then
$x\approx_{Y_2} y $ implies $x\approx_{Y_1} y$. We often identify such $Y$ with the induced subgraph $(Y, R\cap Y^2)$ of $\Gamma$ which it carries. Thus, $\approx_Y$ is $\Aut(Y,R)$-invariant.

Now, each $P_{ij}$ lies in a $\approx_W$-class of $W$. Deleting some sets $P_{ij}$ if necessary (only where elements of distinct sets $P_{ij}$ are $\approx_W$-equivalent, and retaining the assumption that any two distinct elements of $U$ are separated by some set of form $P_{ij}$), we may suppose:  no two elements $x,y$ in distinct sets $P_{ij}, P_{kl}$ are $\approx_W$-equivalent. Also, $\approx_W$-classes contain at most one point of $U$; for if $u_i,u_j\in U$ with $i<j$ then there is a non-empty set $P_{kl}$ whose elements separate $u_i$ and $u_j$, so witness that $u_i\not\approx_W u_j$. Let $m={n\choose 2}$, an upper bound on the number of distinct sets $P_{ij}$. Adjusting the $P_{ij}$ and hence $W$ further, we arrange the sizes of the $P_{ij}$ so that $|P_{ij}|\geq 2$ for each $i,j$ and distinct $\approx_W$-classes of $W$ of size at least two all have different sizes, with size at most $m+1 \in \N$. Now every $\approx_W$-class of $W$ of size greater than 1 consists of a set $P_{ij}$, possibly together with an element of $U$. We will say that a set $Y \subseteq X$ is \emph{huge} if $|Y| > m+1$.

Any automorphism of $(W,R)$ will preserve $\approx_W$, and will fix setwise each $\approx_W$-class of size at least two (as these classes all have different sizes). Hence, if no element of $U$ is $\approx_W$-equivalent to any element of any $P_{ij}$ (that is, elements of $U$ lie in $\approx_W$-classes of size 1), then as the $P_{ij}$ separate the elements of $U$, any automorphism of $W$ will fix $U$ pointwise, as required. So
the concern is that some $\approx_W$-class $C$ in $W$ of size at least two might consist  of a set $P_{ij}$ together with some $u \in U$, in which case there would be  an automorphism of $(W,R)$ mapping $u$ to some vertex  in $C \setminus \{u\}$. 

So suppose $u \in C \cap U$ as in the last paragraph. By the Pigeonhole Principle (retaining all the above reductions, so initially working with larger sets $P_{ij}$) we may suppose for all such $C,u$ that either $u||c$ (that is, $u$ and $c$ are incomparable with respect to $<$) for all $c\in C\setminus\{u\}$, or $u<c$ for all $c\in C\setminus\{u\}$, or $c<u$ for all $c\in C\setminus\{u\}$. 
For  such $u,c$ and $C$, we add a finite set $S_{cu}$ of additional vertices of $\Gamma$ to $W$according to the following recipe.

If $C$ is null, then for each $c \in C \setminus \{u\}$ for which $c \ngtr u$, the set $\Gamma(c) \setminus \Gamma(u)$ is infinite, and we choose $S_{cu} \subset \Gamma(c) \setminus (\Gamma(u)\cup W)$. 
If $C$ is complete, then for each $c \in C \setminus \{u\}$ for which $c \nless u$,  the set $\Gamma(u)\setminus \Gamma(c)$ is infinite, and we choose $S_{cu} \subset \Gamma(u) \setminus (\Gamma(c)\cup W)$. In other cases ($C$ null and $c>u$ for all $c\in C\setminus \{u\}$, or $C$ complete and $c<u$ for all $c\in C\setminus \{u\}$) we do not add any corresponding set $S_{cu}$.
Each $S_{cu}$ (for $u\in U$ and  $c\in W\setminus U$ with $c\approx_W u$) is chosen to be huge, and these sets are chosen so that 
if $(c,u)\neq (c',u')$ then $S_{cu}\cap S_{c'u'}=\emptyset$. We may suppose, again by the Pigeonhole Principle, that for each
such $c,u$, either $c|| x$ for all $x\in S_{cu}$, or $c<x$ for all $x\in S_{cu}$, or $x<c$ for all $x\in S_{cu}$.
By Lemma~\ref{ramsey2} with respect to $L=\{R,<\}$, we may suppose that the collection of all such sets $S_{cu}$ is mutually indiscernible over $W$
(formally, before applying the lemma, the $S_{cu}$ may be taken to be infinite).
Let $V$ be the union of all such sets $S_{cu}$ and of $W$. 
Observe that each $S_{cu}$  is either complete or null, and for each $x,y \in S_{cu}$ and $z \in V \setminus S_{cu}$, we have $x \sim z$ if and only if $y \sim z$. In particular,  any two elements of a set $S_{cu}$ are $\approx_V$-equivalent. We  arrange that all elements of $V\setminus W$ lie in huge $\approx_V$-classes, and that distinct huge $\approx_V$-classes have different sizes, so each is fixed setwise by any automorphism of $(V,R)$.

We aim to show that every automorphism of $(V,R,<)$ must fix $U$ pointwise, which will complete the proof of the lemma. As a first step,  observe  that every huge $\approx_V$-class $S$ contains some set $S_{cu}$. We claim that no huge $\approx_V$-class meets $U$. For suppose $S$ is a huge $\approx_V$-class, with $a\in  U\cap S$. There is $u\in U$ and $c\in W\setminus U$, and a
$\approx_W$-class $C$ with $c,u\in C$,  such that
$S\supset S_{cu}$. Clearly $a=u$, since otherwise, as $S_{cu}$ separates $c$ from $u$,  $a$ would separate $c$ and $u$ and lie in $U\subset W$, contradicting  that $c \approx_W u$. Now if $C$ is null 
 then, by our rule for the process adding $S_{cu}$,  all vertices of $S\setminus \{u\}$ are adjacent to $c$; hence $c$ separates $u$ from other elements of $S$, so $u\not\in S$, a contradiction. Likewise, if $C$ is complete, then all vertices of $S\setminus\{u\}$ are non-adjacent  to $c$, so again $c$ separates $u$ from the rest of $S$, so $u\not\in S$. This proves the claim.
 
\medskip

{\em Claim.} Let $g\in \Aut(V,R,<)$. Then there is $h\in \Aut(V,R,<)_{(U)}$ such that $gh$ fixes $W$ setwise. 

\medskip

{\em Proof of Claim.}
There are distinct (so different-sized) huge $\approx_V$-classes $S_j$ (for $j\in J$), each fixed setwise by $g$,  such that 
$V\setminus W \subseteq \bigcup_{j\in J} S_j$. We may assume that $W$ is not fixed setwise by $g$, as otherwise the claim is trivial. Hence, for some $j\in J$, we have  $(S_j \cap (V\setminus W))^g\neq S_j \cap (V\setminus W)$. 

First, we show that $|S_j\cap W|=1$. 
There are $u\in U$, and some $\approx_W$-class $C$ containing distinct elements $u,c$ of $W$, such that
$S_j\supseteq S_{cu}$. We may suppose that $C$ is null, and $S_{cu}\subset \Gamma(c)\setminus \Gamma(u)$, as the other case where $C$ is complete and 
$S_{cu}\subset \Gamma(u)\setminus \Gamma(c)$ is similar. Now no element of $W\setminus\{u,c\}$ could lie in $S_j$, for otherwise it would separate $u$ from $c$ in $W$ contradicting that $u\approx_W c$. Hence, 
$S_j \cap W\subseteq \{c\}$, so due to the existence of the element $g$, we have $S_j\cap W=\{c\}$. In fact,
$S_j=S_{cu}\cup\{c\}$: for if $c'\in C\setminus \{u,c\}$ then $S_{c'u}\neq \emptyset$ but $c'$ separates elements of $S_{c'u}$ from $c$ so elements of $S_{c'u}$ do not lie in $S_j$; and if $u'\in U\setminus\{u\}$, then no set of form $S_{du'}$ could be a subset of $S_j$, for otherwise $c$ (in $W$) would separate $d$ from $u'$ so the set $S_{du'}$ would not have been added.
 
By our assumption, there is $v\in S_{cu}$ such that $v^g=c$. It is not possible that $S_{cu}$ is totally ordered by $<$; this follows easily from the facts that $g$ induces an automorphism of $(S_{cu}\cup\{c\},<)$,  and the earlier assumption that either $c<x$ for all $x\in S_{cu}$, or $x<c$ for all $x\in S_{cu}$, or $c||x$ for all $x\in S_{cu}$.  
 It follows by Lemma~\ref{ramsey2}(ii) that any permutation of $S_{cu}$, extended by the identity on the rest of $V$, is an automorphism of $(V,R,<)$. In particular distinct elements of $S_{cu}$ are $<$-incomparable, so as $v^g=c$, $S_j$ is an antichain with respect to $<$. Now it could not happen that there is some  $t\in V\setminus S_j$ whose $<$-relation to $c$ is different from its $<$-relation to all other elements of $S_j$. For otherwise $t^{g^{-1}}$ would have a different $<$-relation to $v$ and to all other elements of $S_j$, contradicting the mutual indiscernibility in the construction of $S_{cu}$. It follows that
if $g'$  is the inverse of $g$ on $S_j$ and the identity on the rest of $V$, then $g'\in\Aut(V,R,<)$. The element $h$ of the claim will be a product of elements of the form 
$g'$,  each acting on a different huge $\approx_V$-class.

\medskip

To finish the proof of the lemma, let $g\in \Aut(V,R,<)$, and let $h$ be as in the claim. We must show $u^g=u$ for all $u\in U$. 
Now by construction $gh$ fixes $W$ setwise, and we claim that $gh$ fixes $U$ setwise. 
Indeed, suppose for a contradiction that $u\in U$ and 
 $u^{gh}\not\in U$. As the $\approx_W$-classes of $W$ of size greater than one are all of different sizes, they are all fixed setwise by $gh$.
Hence, 
 as all elements of $W\setminus U$ lie in $\approx_W$-classes of size greater than one, $u^{gh}$ and hence also $u$ lie in some $\approx_W$-class $C$ of size greater than one. Now, by the construction of $V$ from $U$, either 
$u$ is the greatest or least element of $C$ with respect to $<$,  or 
$u$ and $u^{gh}$ are separated by some huge set of form $S_{u,u^{gh}}$. The first case is impossible as $gh$ preserves $<$.  The second case is also impossible, since as the huge $\approx_V$-classes all have different sizes, they are fixed setwise by $gh$. Thus, as claimed, $gh$ induces an automorphism of $(W,R,<)$ which fixes $U$ setwise.
Hence $gh$ fixes $U$ pointwise; for any two distinct elements of $U$ are separated by an $\approx_W$-class of size greater than one, and all such classes have different sizes, so are fixed setwise by $gh$.
  Thus,  $g$ fixes $U$ pointwise.
\end{proof}

\begin{rem} \label{careful}Careful inspection of the above proof shows that if $|U|=n$, then $V$ may be chosen to have size at most $O(n^8)$.  For in constructing $W$ from $U$, if $m={n\choose 2}$ we add at most $m$ sets $P_{ij}$, each of size at least 2 and all of different sizes, so
$|W|= n+k$ where $k:=|W\setminus U|\leq \frac{(m+1)(m+2)}{2}-1$. Then in adding the sets $S_{cu}$ to obtain $V$, we add at most $k$ such sets, each of size at least $m+2$, and all of different sizes. Thus, $|V\setminus W|\leq (m+2)+(m+3)+\ldots+(m+k+1)=\frac{k}{2}(2m+k+3)$. Thus,
$|V|\leq \frac{1}{2}(2n+k(2m+k+5)).$
This is used in Theorem~\ref{stronger} below.
\end{rem}

\begin{proof}[Proof of Proposition~\ref{not2hom}]  By Lemma~\ref{graph}, there is a $G$-invariant graph $\Gamma$ on $X$ such that for all distinct $x,y\in X$, the set $\Gamma(x)\triangle \Gamma(y)$ is infinite.
 The partial order $<$ defined 
in Lemma~\ref{orderedgraph} is clearly also $G$-invariant. The proposition thus follows immediately from that lemma. 
\end{proof}

\section{$G$ $2$-homogeneous but not $2$-transitive}

By Proposition~\ref{not2hom}, to complete the proof of Theorem~\ref{mainthm} it suffices to prove the following.

\begin{prop} \label{not2trans} Let $G$ be a 2-homogeneous but not 2-transitive permutation group on an infinite set $X$. Then the action of $G$ on $X$ is locally rigid.
\end{prop}

Recall that a {\em tournament} is a directed loopless digraph $(T,\to)$ such that for any
distinct vertices $x,y$, exactly one of $x \to y$ or $y\to x$ holds. A group which is $2$-homogeneous but not $2$-transitive has just one orbit on unordered $2$-sets, but two orbits on ordered pairs of distinct elements. Each of these orbits is the arc set of a $G$-invariant tournament with vertex set $X$. Thus, to prove Proposition~\ref{not2trans}, we develop analogues of the methods of Section 2, but for tournaments.

Let $\arc$ denote the arc relation in a tournament $T = (X, \arc)$, and let $G=\Aut(T)$. For $x \in X$, we let $\Gamma^+(x) := \{ y \in X: x \arc y \}$, the set of \emph{outneighbours} of $x$. For $x,y,z \in X$, we say that $z$ \emph{separates} $x,y$ if $x \arc z \arc y$ or $y \arc z \arc x$. Furthermore $Z \subset X$ \emph{separates} $x,y$ if each $z \in Z$ separates $x,y$. We write $x \arc Z$ if $x \arc z$ for each $z \in Z$. 

\begin{prop} \label{cofrigid}
Let $T=(X,\arc)$ be an infinite tournament such that for any distinct $x,y\in X$, 
the sets $\Gamma^+(x)\setminus \Gamma^+(y)$ and $\Gamma^+(y)\setminus \Gamma^+(x)$ are both infinite. Then $T$ is cofinally rigid.
\end{prop}

We first isolate an easy lemma, used to prove Proposition~\ref{cofrigid}, in case it has other uses. It may be  known.

Let $T=(X,\arc)$ be a tournament.
We will say that $A \subset X$ is a \emph{nice} set if $A\neq \emptyset$ and for all $a_1, a_2 \in A$ and $v \in X \setminus A$, we have $a_1 \arc v$ if and only if $a_2 \arc v$. (That is, all vertices in a nice set are related in the same way to vertices outside the nice set; equivalently, no vertex outside a nice set separates a pair of vertices inside the nice set.)
Note that vacuously any singleton is a nice set, and $X$ is nice. 
Furthermore, we will say that $A \subset X$ is a \emph{good} set, if $A$ is totally ordered by $\arc$ and is nice.
We consider the \emph{maximal} good subsets of $X$, that is, good sets $A$ such that there is no good set $A' \subset X$ with $A' \supset A$.  

\begin{lem}\label{maxgood}
If $T=(X,\arc)$ is a tournament, then the maximal good subsets of $X$ form a partition of $X$.
\end{lem}

\begin{proof}
We claim that
if $A$ is good and $B \ne A$ is maximal good (where $A,B\subseteq X$), then either $A \subset B$ or $A \cap B = \emptyset$.
To see this, let
$d \in A \cap B$, and 
let $C = A \cup B$. 
We show that $C$ is good, which ensures $B=C$.

Let $c_1, c_2 \in C, ~v \in X \setminus C$. Now $c_1 \arc v$ if and only if $d \arc v$ if and only if $c_2 \arc v$. This holds because $A$ and $B$ are both nice and $d \in A \cap B$. Hence $C$ is nice.
If $C$ is not totally ordered, then there is some 3-cycle $c_1 \arc c_2 \arc c_3 \arc c_1$ in $C$. Since $A$ and $B$ are both totally ordered, we must have at least one of these points in $A \setminus B$ and one in $B \setminus A$. Suppose $c_1 \in A \setminus B$ and $c_2 \in B \setminus A$ (the other case is similar). Then if $c_3 \in A$, then $c_2$ separates $c_1, c_3$, contradicting the fact that $A$ is nice. Otherwise  $c_3 \in B$, then similarly $c_1$ separates $c_2, c_3$, contradicting the fact that $B$ is nice. Hence $C$ is totally ordered.
Now $B \subseteq C$, and $C$ is good, so $A\subseteq B=C$ by maximality of $B$.

The lemma follows immediately from the claim (using Zorn's Lemma if $X$ is infinite), since each singleton in $X$ is a good set. \end{proof}

\medskip

\begin{proof}[Proof of Proposition~\ref{cofrigid}]
Let $U = \{ u_1, \ldots, u_n \}  \subset X$. 
For any distinct $i,j\in \{1,\ldots,n\}$, the set $\Gamma^+(u_i) \setminus \Gamma^+(u_j) = \{ v \in X: u_i \arc v \arc u_j \}$ is infinite. 
Hence by Ramsey's Theorem, there is $U_{ij} \subseteq \Gamma^+(u_i) \setminus (\Gamma^+(u_j)\cup\{u_j\})$ with $|U_{ij}| = \aleph_0$, such that $U_{ij}$ is totally ordered by $\arc$. 
Note that the sets $U_{ij}, U_{ji}$ both separate $u_i, u_j$ (since $u_i \arc U_{ij} \arc u_j$, and $u_j \arc U_{ji} \arc u_i$). We may choose the $U_{ij}$ so that if $(i,j)\neq (k,l)$ then $U_{ij}\cap U_{kl}=\emptyset$.

\medskip

{\em Claim 1.}
Let $N$ be any positive integer. Then there are finite subsets $V_{ij}$ of $U_{ij}$ (for all distinct integers $i,j$ with $1\leq i,j\leq n$) of size  $N$ such that the following holds, where $T'$ is the induced subtournament of $T$ with vertex set
$ U \cup \bigcup_{i\neq j} V_{ij}$: for any distinct $i,j\in \{1,\ldots,n\}$, and  for each $x,y \in U_{ij}$ and $v \in T' \setminus U_{ij}$, $x \arc v$ if and only if $y \arc v$.

\medskip

{\em Proof of Claim 1.}
This is an immediate  application of Lemma~\ref{ramsey2}.

\medskip

 Provided we initially choose $N$ large enough, we may cut the $V_{ij}$ down further, and so
 suppose that each set $V_{ij}$ has size exactly $2^r$ for some $r\geq 2$, and that distinct sets $V_{ij}$ and $V_{kl}$ have distinct sizes. Observe (for use in Theorem~\ref{stronger}) that $T'$ has $n+\Sigma_{i=2}^{m+1} 2^i$ vertices where
$m=2{n\choose 2}$, that is, it has  $n+2^{n^2-n+2}-2$ vertices.
We claim that $T'$ is rigid, which suffices to prove the lemma. Let $V$ denote the vertex set of $T'$ (a union of $U$ and the sets $V_{ij}$).

The sets $V_{ij}$ are clearly all good, though possibly not maximal good. Hence, by Lemma~\ref{maxgood}, if $B \cap V_{ij} \ne \emptyset$ and $B$ is maximal good, then $V_{ij} \subseteq B$. 

The idea of the proof is as follows. 
First observe that automorphisms of the subtournament $(V,\to)$ of $T$ preserve the family of maximal good sets. 
We aim to show that by our construction of $V$, all non-singleton maximal good sets in $V$ have different sizes, so in fact each is fixed setwise, and hence pointwise, by any automorphism. 
We then show that if some automorphism $\alpha$ of $(V,\to)$ fixes pointwise all non-singleton maximal good subsets of $V$, then $\alpha$ fixes $V$ pointwise.

\medskip

{\em Claim 2.}
If $A$ is a good subset of $V$, then $|A \cap U| \le 1$.

\medskip

{\em Proof of Claim 2.}
Suppose $u_1, u_2 \in A \cap U$, with $u_1 \ne u_2$. 
We have $u_1 \arc V_{12} \arc u_2$. Since $A$ is good, we must have $V_{12} \subset A$: otherwise any $y \in V_{12} \setminus A$ separates $u_1, u_2$, contradicting the fact that $A$ is nice. 
Similarly, we have $u_2 \arc V_{21} \arc u_1$, and we must have $V_{21} \subset A$.
But then we have $\{u_1, u_2\} \cup V_{12} \cup V_{21} \subseteq A$, and $u_1 \arc V_{12} \arc u_2 \arc V_{21} \arc u_1$. But then  $A$ is not totally ordered by $\arc$, which contradicts the fact that $A$ is good.

\medskip

Thus, maximal good sets are unions of sets $V_{ij}$ with at most one element of $U$ added (this includes the case of a singleton point of $U$).
Then by our choice of the sizes of the $V_{ij}$ in the construction, any two non-singleton maximal good sets have different sizes. 
(For let the $V_{ij}$ have sizes $n_1, \ldots, n_t$, say. These were chosen as distinct powers of 2, and  so  all numbers of the form $n_{i_1} + \ldots + n_{i_s}$ or $n_{i_1} + \ldots + n_{i_s} +1$ are distinct.)
Hence any automorphism of $V$ fixes each non-singleton maximal good set setwise, and hence also pointwise since each is totally ordered and so rigid. Thus any automorphism fixes all elements of $V \setminus U$ pointwise, and so also fixes $U$ pointwise; indeed, for each pair of elements of $U$ there is some $Z \subset V \setminus U$ separating the pair, and so no automorphism can move points of $U$. 
\end{proof}

\begin{cor} \label{coro} Let $T$ be an infinite tournament with 2-homogeneous automorphism group. Then $T$ is cofinally rigid.
\end{cor}

\begin{proof} By Ramsey's Theorem, there is a subtournament of $T$ of the form $\{ x_i: i \in \N \}$ with $x_i \arc x_j$ if and only if $i < j$ (or possibly with all arcs reversed). 
Clearly, if $i<j$, then $|\Gamma^+(x_j)\triangle \Gamma^+(x_i)|\geq j-i-1$. 
By 2-homogeneity of $G$ (and therefore $\Aut(T)$), there is $d \in \N \cup \{ \aleph_0 \}$ such that if $x \ne y$ then $|\Gamma^+(x) \triangle \Gamma^+(y)| = d$.
Hence, $d \ge n$ for each $n \in \N$, so $d = \aleph_0$.

We may suppose that $T$ is not totally ordered by $\arc$, since finite total orders are rigid. By Proposition~\ref{cofrigid}, the proof of the corollary now reduces to the following claim.

\medskip

{\em Claim.}
For all distinct $x,y \in X$, the sets $\Gamma^+(x) \setminus \Gamma^+(y)$ and  $\Gamma^+(y) \setminus \Gamma^+(x)$ are both infinite. 

\medskip

{\em Proof of Claim.}
Suppose that for some $u,v \in X$ with $u \neq v$, the set $\Gamma^+(u) \setminus \Gamma^+(v)$ is infinite, but $\Gamma^+(v) \setminus \Gamma^+(u)$ is finite. 
Now, using 2-homogeneity, define an order relation $<$ on $X$, such that  $x < y$ if and only if $\Gamma^+(x) \setminus \Gamma^+(y)$ is infinite. 
This is a $G$-invariant partial order on $X$, containing comparable pairs. 
By 2-homogeneity, it follows that $<$ is a total order, and it or its reverse agrees with $\arc$. This contradicts the above assumption.
 \end{proof}

\medskip

\begin{proof}[Proof of Proposition~\ref{not2trans}] As noted above, there is a $G$-invariant tournament $T$ with vertex set $X$, whose arc set is a $G$-orbit on $X^{[2]}$. The proposition now follows immediately from 2-homogeneity and Corollary~\ref{coro}. 
\end{proof}

\medskip

\begin{proof}[Proof of Proposition~\ref{mainthm}] This is immediate from Lemma~\ref{locallyrigid} and Propositions~\ref{not2hom} and \ref{not2trans}. 
\end{proof}

\section{Further remarks}

The proof of Theorem~\ref{mainthm} yields that there is a function $f:{\mathbb N}\to {\mathbb N}$ such that for any $l\in {\mathbb N}$, if $H\leq G$ are closed permutation groups on an infinite set $X$ with $G$ primitive but not 2-transitive, and $G$ and $H$ have the same orbits on $X^{[n]}$ for all $n\leq f(l)$, then $G$ and $H$ have the same orbits on $X^m$ for all $m\leq l$. An upper bound for $f$ is given by the cardinality of $V$ in terms of $|U|$ in the definition in the Introduction of a  group $G$ acting locally rigidly. By the proofs of Propositions~\ref{not2hom} and \ref{not2trans}, we 
obtain the following slight strengthening of Theorem~\ref{mainthm}, probably far from best possible. Observe that with $m$ and $k$ as in Remark~\ref{careful},
$\frac{1}{2}(2n+k(2m+k+5))\leq n+2^{n^2-n+2}-2$ for all $n>1$, so the bound in the proof of
Proposition~\ref{not2trans} dominates.

\begin{thm}\label{stronger}
Let $G,H$ be closed permutation groups on the infinite set $X$, with $G$ primitive but not 2-transitive on $X$, and with $H\leq G$. Let $n\in {\mathbb N}$, and suppose that $G$ and $H$ have the same orbits on the set $X^{[l]}$ for each $l\leq n+2^{n^2-n+2}-2$. Then $G$ and $H$ have the same orbits on $X^m$ for each $m\leq n$.
\end{thm}

Theorem~\ref{mainthm} requires the assumption of primitivity. For example, 
$\Aut({\mathbb Q},<) {\rm Wr} C_2$ is orbit equivalent to $\Sym({\mathbb Q}) {\rm Wr} C_2$ (in the natural imprimitive action). However, a proof of Conjecture~\ref{mainconj} should yield a lot of information about the imprimitive case.

A proof of Conjecture~\ref{mainconj}, at least if via local rigidity, would appear to require arguments considerably more involved than those of this paper. As an example, suppose that $G$ is 2-primitive (that is, 2-transitive and with  primitive point stabilisers) but not 3-homogeneous on the infinite set $X$. We conjecture that $G$ acts locally rigidly. There is a $G$-invariant 3-hypergraph $\Gamma$ on $X$, and we would like to show that $\Gamma$ (possibly expanded by some other $G$-invariant relations) is locally rigid. Given $x\in X$, there is an induced  graph 
$\Gamma_x$ on $X\setminus \{x\}$ on which $G_x$ acts primitively. However, it is not clear that local rigidity of $\Gamma_x$ transfers to local rigidity of $\Gamma$, or that a straightforward induction on the degree of transitivity of $G$ can be made to work.
There may also be an approach to local rigidity of hypergraphs using  \cite[Lemma 2.5]{macpherson1} and related results.

We cannot even prove the conjecture under the assumptions that $X$ is countable and $G$ is locally compact (that is, there is some finite $F\subset X$ such that all orbits of $G_{(F)}$ on $X$ are finite). Even the case when $G$ is countable is open. A first class to consider would be that of primitive groups with finite point stabiliser, for which Smith \cite{smith} gives a useful-looking  version of the O'Nan-Scott Theorem.

However, as evidence for the conjecture, we observe that an obvious place  to look for a counterexample, suggested by the family of closed supergroups of $\Aut({\mathbb Q},<)$ listed in Conjecture~\ref{mainconj}, fails. Indeed, let $(T,<)$ be any of the  countable 2-homogeneous trees (that is, semilinear orders) classified by Droste in \cite{droste}. There is a family of interesting primitive closed permutation groups associated with $\Aut(T,<)$, namely the primitive Jordan permutation groups with primitive Jordan sets classified in \cite{an}: we have in mind $\Aut(T,<)$, the  automorphism group of the ternary general betweenness relation on $T$ induced from $<$, the automorphism group of the corresponding countable $C$-structure, a structure whose elements are a dense set of maximal chains in $(T,<)$, and the automorphism group of the corresponding $D$-relation  (a quaternary relation on the set of `directions' of the betweenness relation). It can be checked that each of these groups acts locally rigidly. We omit the details.

In \cite{volta} a permutation group $G$ on $X$ is defined to be {\em orbit-closed} if there is no $H\leq \Sym(X)$ which properly contains $G$ and is orbit-equivalent to $G$. Such $G$ will be a closed permutation group, and Conjecture~\ref{mainconj} asserts that if $X$ is countably infinite then the only primitive closed permutation groups which are not orbit-closed are the proper subgroups of $\Sym(X)$ listed in that conjecture.
In \cite{volta} the authors  define $G\leq \Sym(X)$ to be a {\em relation group}  if there is a collection $R$ of finite subsets of $X$ such that 
$$G=\{g\in \Sym(X): \forall a \in {\cal P}(X)(a \in R \leftrightarrow a^g \in R)\}.$$
Clearly any relation group is orbit-closed. Also,  by \cite[Corollary 4.3]{volta}, any {\em finite} primitive orbit-closed group is a relation group. We do not know  whether this holds without finiteness, and in particular cannot answer  the following question, to which Siemons drew our attention.

\begin{question}\label{siemons}
Is $\Aut({\mathbb Q},<)$ the only primitive but not 2-transitive closed permutation group of countable degree which is not a relation group?
\end{question}

\noindent
As a small example, let $\Gamma_3$ be the universal homogeneous 2-edge-coloured graph with edges coloured randomly red or green; that is, the unique countably infinite homogeneous 2-edge-coloured graph such that for any three finite disjoint sets $U,V,W$ of vertices, there is a vertex $x$ not adjacent to any vertex in $U$, adjacent by a red edge to each element of $V$ and by a green edge to each element of $W$. At first sight, $G=\Aut(\Gamma_3)$ is not a relation group, but in fact it is a relation group; for  we may  take $R$ to consist of the 2-sets joined by a red edge and the 3-sets which carry a green triangle.

Our remarks in the Introduction suggest a further question. Again, for convenience, we shall consider actions on a countably infinite set $X$. A subset $Y$ of $X$ is a {\em moiety} of $X$ if $|Y|=|X\setminus Y|$.

\begin{question}
Which primitive closed permutation groups $G$ on a countably infinite set $X$ have a regular orbit on moieties?
\end{question}

To say that $G$ has a regular orbit on moieties of $X$ is the same as to say, in the language of \cite{laflamme}, that any first order structure $M$ on $X$ with $G=\Aut(M)$ has {\em distinguishing number 2}. Some results on this are obtained in \cite{laflamme}. For example, if $M$ is a homogeneous structure such that the collection of finite structures which embed in it is a `free amalgamation class', and $\Aut(M)$ is primitive but for some $k$ is not $k$-transitive, then $\Aut(M)$ has a regular orbit on moieties. In particular, this holds for the random graph, as follows already from \cite[Theorem 3.1]{henson}. 
On the other hand, as noted in \cite{laflamme} it is easily seen that $\Aut({\mathbb Q},<)$ has no regular orbit on moieties; for if $A$ is a moiety of ${\mathbb Q}$ whose setwise stabiliser is trivial, then $A$ is dense and codense in ${\mathbb Q}$, but the structure 
$({\mathbb Q},<, P)$, where $P$ is a unary predicate naming a dense codense set,
is homogeneous so admits $2^{\aleph_0}$ automorphisms.

This suggests the following strengthening of orbit-equivalence. Let us say that
permutation groups $G,H$ on the countably infinite set $X$ are {\em strongly orbit-equivalent} if they have the same orbits on the power set ${\cal P}(X)$ of $X$ (not just on {\em finite} subsets of $X$).
The following conjecture is implied by Conjecture~\ref{mainconj}, for it is easily seen that the five closed groups containing $\Aut({\mathbb Q},<)$ all have different orbits on ${\cal P}({\mathbb Q})$. For example, $\Aut({\mathbb Q},<)$ has an orbit consisting of increasing subsets of order type $\omega$ with rational supremum, but this family of sets is not invariant under the automorphism groups of the induced circular order or linear betweenness relation.

\begin{conj} Let $G,H$ be strongly orbit-equivalent closed permutation groups on the countably infinite set $X$. Then $H=G$.
\end{conj}

Again, the assumption that the groups are closed is necessary. Stoller (\cite{stoller}, see also \cite{neumann}) gives an example of a proper subgroup $H$ of $G=\Sym({\mathbb N})$ which is strongly orbit-equivalent to $G$; namely, let $H$ consist of those permutations $g$ of ${\mathbb N}$ such that there are  two partitions, dependent on $g$,  of ${\mathbb N}$ into finitely many sets $A_1,\ldots,A_k$ and $B_1,\ldots,B_k$ (so ${\mathbb N}=A_1\cup\ldots\cup A_k=B_1\cup\ldots \cup B_k$, each partitions) such that for each $i=1,\ldots,k$, the element $g$ induces an order isomorphism $(A_i,<)\to (B_i,<)$.

Finally, we mention a conjectural strengthening of Lemmas~\ref{prop6.1} and \ref{orderedgraph}. It is  a special case of a much stronger conjecture in \cite{sethomg}.

\begin{conj}
Let $\Gamma$ be an infinite graph such that for any distinct vertices $x,y$ the set $\Gamma(x)\triangle \Gamma(y)$ is infinite. Then $\Gamma$ is locally rigid.
\end{conj}

{\em Acknowledgement.} We thank J. Siemons for drawing attention to \cite{volta} and \cite{smith} and to Question~\ref{siemons}.

\end{document}